\documentclass{amsart}

\newtheorem{theorem}{Theorem}[section]
\newtheorem{lemma}[theorem]{Lemma}

\newtheorem{corollary}[theorem]{Corollary}

\theoremstyle{definition}
\newtheorem{definition}[theorem]{Definition}

\theoremstyle{remark}

\numberwithin{equation}{section}

\begin{document}
\title[A generalized weighted Hardy-Ces\`{a}ro operator, and its commutator]{A generalized weighted Hardy-Ces\`{a}ro operator, and its commutator on weighted $L^p$ and BMO spaces}

\author{Nguyen Minh Chuong}
\address{Institute of mathematics, Vietnamese  Academy of Science and Technology,  Hanoi, Vietnam.}
\email{nmchuong@math.ac.vn}
\author{Ha Duy Hung}
\address{Hanoi National University of Education, Hanoi, Vietnam.}
\email{hunghaduy@gmail.com}

\thanks{This paper is granted by Vietnam NAFOSTED (National Foundation for Science and Technology Development).}


\begin{abstract}In this paper, we introduce a new weighted Hardy-Ces\`{a}ro operator defined by $U_{\psi,s}f(x)=\int\limits_0^1 f\left(s(t)\cdot x\right) \psi(t)dt$, which is associated to the parameter curve $s(t,x)=s(t)x$. Under certain conditions on $s(t)$ and on an absolutely homogeneous weight function $\omega$, we characterize the weight function $\psi$ such that $U_{\psi,s}$ is bounded on $L^p(\omega)$, $BMO(\omega)$. The corresponding operator norms are worked out too. These results extend the ones of Jie Xiao \cite{xiao}. We also give a sufficient and a necessary condition on the weight function $\psi$, which ensure the boundedness of the commutators of operator $U_{\psi,s}$ on $L^p(\omega)$ with symbols in $BMO(\omega)$. 
\end{abstract}

\maketitle

\section{Introduction}

Let $f\in L^1_{\rm loc}$, then the classical Hardy operator $U$ is defined by 
\[
Uf(x)=\frac1x\int_0^x f(t)dt.
\]
A celebrated Hardy integral inequality can be formulated as
\[
\|Uf\|_{L^p(\mathbb R)}\leq \frac p{p-1}\|f\|_{L^p(\mathbb R)},
\]
where $1<p<\infty$ and the constant $\frac p{p-1}$ is the best possible. 
\vskip12pt
The Hardy integral inequality and its variants play an important role in various branches of analysis such as approximation theory, differential equations, theory of function spaces etc. Therefore, there are various papers studying Hardy integral inequalities for operator $U$ and its generalizations. Up to now, there are two types of Hardy operator in $n-$dimension case. The first one was introduced by Faris \cite{faris} in 1976, and in 1995, M. Christ and L. Grafakos \cite{christ} gave an equivalent version of $n-$dimensional Hardy operator
\begin{equation}\label{eq2}
{\mathcal H}f(x)=\frac1{\Omega_n|x|^n}\int_{|y|<|x|}f(y)dy,
\end{equation}
where $\Omega_n=\frac{\pi^{n/2}}{\Gamma(1+n/2)}$.	

The second one appeared in 1984, by Carton-Lebrun and Fosset \cite{carton1}, in which the authors defined the weighted Hardy operator $U_\psi$ as the following. Let $\psi: [0, 1]\to[0,\infty)$ be a measurable function, and let $f$ be a measurable complex-valued function on $\mathbb R^n$. The weighted Hardy operator $U_\psi$ is defined formally by 
\begin{equation}\label{eq1}
U_\psi f(x)=\int_0^1f(tx)\psi(t)dt,\qquad x\in\mathbb R^n.
\end{equation}
\vskip12pt
Various approaches are described and some extensions are given for the first Hardy operator $U$ and its modifications. Specially, there appeared a lot of papers which have discussed the problems of characterizing the weights $(u,v)$, for which $U$ and its generalizations are of weak and strong type $(p,q)$, or are bounded between Lorentz spaces, BMO, ... (see \cite{andersen1,bradley,chuong1, chuong2, chuong3, chuong4, edmunds, muckenhoupt3, lai, reyes1, sawyer, stein1, stein2, stepanov}). 
\vskip12pt
Under certain conditions on $\psi$, Carton-Lebrun and Fosset \cite{carton1} found that $U_\psi$ is bounded from $BMO(\mathbb R^n)$ into itself. Moreover,  $U_\psi$ commutes with the Hilbert transform in the case $n=1$ and with a certain Calder\'{o}n-Zygmund singular integral operator (and thus with the Riesz transform) in the case $n\geq2$. In 2001, J. Xiao \cite{xiao} obtained that $U_\psi$ is bounded on $L^p(\mathbb R^n)$ if and only if 
\[
\mathcal A:=\int_0^1t^{-n/p}\psi(t)dt<\infty.
\]
Meanwhile, the corresponding operator norm was worked out. The result seems to be of interest as it is related closely to the Hardy integral inequality. For example, if $\psi\equiv 1$, and $n=1$, then $U_\psi$ may be reduced to the classical Hardy operator. In \cite{xiao}, J. Xiao also obtained the $BMO(\mathbb R^n)-$bounds of $U_\psi$, which sharpened and extended the main result of Carton-Lebrun and Fosset in \cite{carton1}. He considered the weighted Ces\`{a}ro operator 
\[
V_{\psi}f(x)=\int_0^1f\left(x/t\right)t^{-n}\psi(t)dt.
\]
It is known that both Hardy operator and Ces\`{a}ro operator play an important role in various fields. For example, the Calder\'{o}n maximal operator (see \cite{bennet}), which is important in interpolation theory, is the sum of the classical Hardy operator $U$ and  the classical Ces\`{a}ro operator $V$, where
\[
Vf(x)=\int_x^\infty \frac{f(y)}ydy,\quad(x>0).
\]
In fact, $V_\psi$ is Banach adjoint of $U_{\psi}$ and they are commutative and thus J. Xiao also obtains the boundedness and bounds of $V_\psi$ on $L^p$, $BMO$ and $H^1(\mathbb R^n)$-the Hardy space on $\mathbb R^n$.  
\vskip12pt
Recently, Z. W Fu, Z. G Liu, and S. Z Lu \cite{fu1} give  a sufficient and necessary condition on weight function $\psi$, which ensures the boundedness of the commutators of weighted Hardy operators $U_\psi$, with symbols in $BMO(\mathbb R^n)$ and $L^p(\mathbb R^n)$ with $1<p<\infty$. In addition, several authors have considered the boundedness and bounds of $U_\psi$ on Morrey spaces, Campanato spaces, $Q^{\alpha}_{p,q}$-type spaces, Triebel-Lizorkin-type spaces (see \cite{kuang,tang1,tang2,zhao}). We note that in \cite{samko1,samko2}, the authors showed a very interesting application of the second weighted Hardy operators. They notes that the boundedness of the Cauchy singular integral operator $S_\Gamma$ in Morrey spaces on an arbitrary Calerson curve could be reduced to the boundedness of weighted Hardy operators to be bounded in Morrey spaces. These Hardy operators are the ones of second type.
\vskip12pt
There is a connection between the above two types of Hardy operators. If $\psi(t)=1$, and $n=1$, then $U_{\psi}$ is just reduced to the classical Hardy operator $U$. If $n\geq2$, then the restriction of $\mathcal H$ on class of radial functions is $U_\psi$ with $\psi(t)=nt^{n-1}$ (see \cite{christ,zhao}). Thus, many results on the Hardy operators of second type (which seems to be more easier to study) can be changed to the first one with restriction on class of radial functions. These Hardy operators of second type also contain the class of Riemann-Liouville integral operators in case $n=1$ and $\psi(t)=\beta(1-t)^{\beta-1}$, $\beta>0$. We obverse that, the value of $U_\psi f$ at a point $x$ just depends on the weight average value of $f$ along a parameter $s(t,x)=tx$. By the analogues of singular and maximal operators associated with certain submanifolds of positive codimension  in $\mathbb R^n$, we are motivated to consider the  generalized Hardy-Ces\`{a}ro operator as follows.
 \begin{definition}Let $\psi:[0,1]\to[0,\infty)$, $s:[0,1]\to\mathbb R$  be measurable functions. We define the generalized Hardy-Ces\`{a}ro operator $U_{\psi,s}$, associated to the parameter curve $s(x,t):=s(t)x$, as
\begin{equation}\label{eq3}
U_{\psi,s}f(x)=\int_0^1f\left(s(t) x\right)\psi(t)dt,
\end{equation}
for a measurable complex valued function $f$ on $\mathbb R^n$.
\end{definition}
Let us explain why we call $U_{\psi,s}$  Hardy-Ces\`{a}ro operators. In fact, they contain both type of classical Hardy operator and Ces\`{a}ro operator. 
If $s(t)=t$, $U_{\psi,s}$ is reduced to $U_\psi$ and if $s(t)=1/t$, we replace $\psi(t)$ by $t^{-n}\psi(t)$, then $U_{\psi,s}$ is reduced to weighted Ces\`{a}ro operator 
\[
V_\psi f(x)=\int_0^1f\left(x/t\right)t^{-n}\psi(t)dt.
\]

The operator $V_\psi$ can be generalized to
\begin{equation}\label{eq3a}
V_{\psi,s}f(x)=\int_0^1f\left(s(t)x\right)|s(t)|^{n}\psi(t)dt.
\end{equation}
\vskip12pt
Let $b$ be a locally integrable function on $\mathbb R^n$. The commutators of $b$ and operators $U_{\psi,s}, V_{\psi,s}$ are respectively defined by
\begin{equation}\label{eq4}
U^b_{\psi,s}f=bU_{\psi,s}(f)-U_{\psi,s}(bf),
\end{equation}
and
\begin{equation}\label{eq5a}
V^b_{\psi,s}f(x)=bV_{\psi,s}(f)-V_{\psi,s}(bf).
\end{equation}

\vskip12pt
For convenience, let us introduce some definitions and notations. Let $\omega(x)$ be a measurable function in $\mathbb R^n$ with $\omega\geq0$. A measurable function $f$ is said to be in $L^p(\omega)$ if
\[
\|f\|_{L^p(\omega)}=\left(\int_{\mathbb R^n}|f(x)|^p\omega(x)dx\right)^{1/p}<\infty.
\]

We denote by $BMO(\omega)$ the space of all functions $f$, which are of bounded mean oscillation with weight $\omega$, that is
\begin{equation}\label{eq5b}
\|f\|_{BMO(\omega)}=\sup\limits_B\frac1{\omega(B)}\int_B|f(x)-f_{B,\omega}|\omega(x)dx<\infty,
\end{equation}
where supremum is taken over all $n-$dimensional ball $B$. Here, $\omega(B)=\int_B\omega(x)dx$, and $f_{B,\omega}$ is the mean value of $f$ on $B$ with weight $\omega$:
\[
f_{Q,\omega}=\frac1{\omega(Q)}\int_Qf(x)\omega(x)dx.
\]
The case $\omega\equiv1$ of (\ref{eq5b}) corresponds to the class of functions of bounded mean oscillation of F. John and L. Nirenberg \cite{john}. We obverse that $L^\infty(\mathbb R^n)\subset BMO(\omega)$.  It is useful to state the following well-known result:
\begin{lemma}\label{lem1a} Let $\omega$ be a weight function with doubling property, that is for some positive constant $A$, we have
\[
\omega\left(B(x,2r)\right)\leq C\omega\left(B(x,r)\right),
\]
for all $x\in \mathbb R^n$ and $r>0$. Then, for any $1<p<\infty$, there exists some positive contant $C_p$ so that
\[
\|f\|_{BMO^p(\omega)}=\left(\frac1{\omega\left(B\right)}\int_B\left|f(x)-f_{B,\omega}\right|^p\omega(x)dx\right)^{1/p}\leq C_p\|f\|_{BMO(\omega)}.
\]
\end{lemma}
\vskip12pt
The Hardy-Littlewood maximal operator $M_\omega$ with weight $\omega$ is defined by 
\[
M_\omega f(x)=\sup\limits_B\frac1{\omega(B)}\int_B|f(y)|\omega(y)dy.
\]

There is a simple variant of the maximal theorem for weight as the following
\begin{lemma}\label{lem1c} Let $\omega$ has doubling property, then there exists a positive constant $C_r$ so that 
\[
\int_{\mathbb R^n}\left(M_\omega f(x)\right)^r\omega (x)dx\leq C_r\int_{\mathbb R^n}|f(x)|^r\omega(x)dx
\]
for any $1<r<\infty$, $f\in L^r(\omega)$.

\end{lemma}
Lemmas \ref{lem1a} and \ref{lem1c} are well-known (see \cite[Chapter V]{stein2}), so we omit their proofs.
\vskip12pt
It is our goal in this paper to study weighted norm inequalities for weighted Hardy-Ces\`{a}ro operator $U_{\psi,s}$ and its commutator $U^b_{\psi,s}$.  More precisely, we obtain some sufficient conditions on weight function $\omega$, and on function $s(t)$ for which we give the sufficient and necessary conditions on $\psi$ so that  $U_{\psi,s}$ is bounded on $L^p(\omega)$ and $BMO(\omega)$. It is interesting to know that, in some case of $s(t)$, the $BMO-$bound of $U_{\psi,s}$ depends on the dimension $n$. We also find the characterizations on $\psi(t)$ so that, under certain conditions on $\omega(x)$ and $s(t)$, then the commutator $U^b_{\psi,s}$ of the generalized Hardy-Ces\`{a}ro operator $U_{\psi,s}$ is bounded on $BMO(\omega)$ with symbol $b\in BMO(\omega)$. These results actually are more general than those obtained in \cite{xiao}, \cite{fu1}.

\section{Bounds of Hardy-Ces\`{a}ro operator on weight $L^p$ and weight BMO spaces}
The purpose of this section is to prove the boundedness of the generalized Hardy-Ces\`{a}ro operator $U_{\psi,s}$ on $L^p(\omega)$ and $BMO(\omega)$ with certain conditions on $\omega$. The method involve techniques similar to those used in \cite{xiao}, but with a more careful analysis for the boundedness on $BMO(\omega)$. We also work out the operator norms on such weight spaces. These results extend the results in \cite{xiao}.
\vskip12pt
 It is conventional that $\int_{S_n}\omega(x)d\sigma(x)$ refers to $2\omega(1)$ in case $n=1$. Then, the class of weight functions $\omega$, which we shall consider, is the following.
\begin{definition}
Let $\alpha$ be a real number. Let $\mathcal W_{\alpha}$ be the set of all functions $\omega$ on $\mathbb R^n$, which are measurable, $\omega(x)>0$ for almost everywhere $x\in\mathbb R^n$, $0<\int_{S_n}\omega(y)d\sigma(y)<\infty$, and are absolutely homogeneous of degree $\alpha$, that is $\omega(tx)=|t|^\alpha\omega(x)$, for all $t\in\mathbb R\setminus\{0\}$, $x\in\mathbb R^n$, where $S_n=\{x\in\mathbb R^n:\;|x|=1\}$.
\end{definition}
\vskip12pt
We remind that, if we define the measure $\rho$ on $(0,\infty)$ by $\rho(E)=\int_E r^{n-1}dr$, and the map $\Phi(x)=\left(|x|,\frac x{|x|}\right)$, then there exists a unique Borel measure $\sigma$ on $S_n$ such that $\rho\times\sigma$ is the Borel measure induced by $\Phi$ from Lebesgue measure on $\mathbb R^n$ ($n>1$). (see \cite[page 78]{folland}, \cite[page 142]{krantz} for more details). \vskip12pt
Let us describe some typical examples and properties of $\mathcal W_\alpha$. Note that, a weight $\omega\in\mathcal W_\alpha$ may not need to belong $L^1_{\rm loc}(\mathbb R^n)$. In fact, we observe that if $\omega\in \mathcal W_\alpha$, then $\omega\in L^1_{\rm loc}(\mathbb R^n)$ if and only if $\alpha>-n$. 
\vskip12pt
 If $n=1$, then $\omega(x)= c|x|^\alpha$, for some positive constant $c$. For $n\geq1$ and $\alpha\neq0$, $\omega(x)=|x|^\alpha$ is in $\mathcal W_\alpha$. If $\omega_1,\omega_2$ is in $\mathcal W_\alpha$, so is $\theta\omega_1+\lambda\omega_2$ for all $\theta,\lambda>0$. There are many other examples in case $n>1$ and $\alpha\neq0$, namely $\omega(x_1,\ldots,x_n)=|x_1|^\alpha$. In case $n>1$ and $\alpha=0$, we can construct a non-trivial example as the following: let $\phi$ be any positive, even and locally integrable function on $S_n=\{x\in\mathbb R^n:\;|x|=1\}$, then
\[
\omega(x)=
\begin{cases} 
\phi\left(\frac x{|x|}\right)\qquad\text{if $x\neq0$},&\\
\quad0\quad\;\;\;\qquad\text{if $x=0$},&\\
\end{cases}
\]
is in $\mathcal W_0$.


\begin{lemma}\label{lem1} For any real number $\alpha$, if $\omega\in \mathcal W_\alpha$, and $\epsilon>0$, then
\[
\int_{|x|>1}\frac{\omega(x)}{|x|^{n+\alpha+\epsilon}}dx=\int_{|x|<1}\frac{\omega(x)}{|x|^{n+\alpha-\epsilon}}dx=
\begin{cases}
\frac{\int_{S_n}\omega(x)d\sigma(x)}{\epsilon}\quad\text{\rm if}\;n>1,\vspace*{8pt}&\\ 
\quad\frac{2\omega(1)}{\epsilon}\quad\qquad\text{\rm if}\;n=1.\end{cases}
\]
\end{lemma}
\begin{proof}Since the proofs of these equalities are similar, we shall only give the proof of 
\[
\int_{|x|>1}\frac{\omega(x)}{|x|^{n+\alpha+\epsilon}}dx=
\begin{cases}
\frac{\int_{S_n}\omega(x)d\sigma(x)}{\epsilon}\quad\text{\rm if}\;n>1,\vspace*{8pt}&\\ 
\quad\frac{2\omega(1)}{\epsilon}\quad\qquad\text{\rm if}\;n=1\end{cases}
\]
We first consider the case $n>1$. 
Then
\[
\int_{|x|>1}\frac{\omega(x)}{|x|^{n+\alpha+\epsilon}}dx=\int_1^\infty dr\int_{S(0,r)}\frac{\omega(y)}{r^{n+\alpha+\epsilon}} d\sigma(y).
\]
Since $\omega\in\mathcal W_\alpha$, we have
\[
\int_{|x|>1}\frac{\omega(x)}{|x|^{n+\alpha+\epsilon}}dx=\int_1^\infty \frac1{r^{n+\epsilon}}dr\cdot \int_{S_n}\omega(y)d\sigma(y)=\frac1\epsilon {\int_{S_n}\omega(x)d\sigma(x)}.
\]
We now consider the case $n=1$. Since $\omega\in\mathcal W_\alpha$, we have $\omega(x)=c|x|^\alpha$, for some contant $c>0$. We have
\[
\int_{|x|>1}\frac{\omega(x)}{|x|^{1+\alpha+\epsilon}}dx=\int_1^{\infty}\frac1{x^{1+\alpha+\epsilon}}\left(\omega(x)+\omega(-x)\right)dx=\frac{2c}\epsilon.
\]
\end{proof}
\vskip12pt
In \cite{xiao}, Jie Xiao proved that $U_\psi$ is bounded on $L^p(\mathbb R^n)$, if and only if $$\mathcal A=\int_0^1t^{-n/p}\psi(t)dt<\infty.$$ He also showed that $\mathcal A$ is the $L^p$-operator norm   of $U_\psi$. One of our main results in this section is formulated as follows.
\begin{theorem}\label{theo1} Let $p\in[1;\infty]$, $\alpha$ be real numbers and $\omega\in \mathcal W_{\alpha}$. Let $s:[0;1]\to\mathbb R$ be a measurable function so that $|s(t)|\geq t^\beta$ a.e. $t\in[0,1]$, for some constant $\beta>0$. Then, $U_{\psi,s}:L^p(\omega)\to L^p(\omega)$ exists as a bounded operator if and only if 
\begin{equation}\label{eq6}
\int_0^1|s(t)|^{-\frac{n+\alpha}p}\psi(t)dt<\infty.
\end{equation}
Moreover, when (\ref{eq6}) holds, the operator norm of $U_{\psi,s}$ on $L^p(\omega)$ is given by
\begin{equation}\label{eq7}
\|U_{\psi,s}\|_{L^p(\omega)\to L^p(\omega)}=\int_0^1|s(t)|^{-\frac{n+\alpha}p}\psi(t)dt.
\end{equation}
\end{theorem}
\vskip12pt

\begin{proof}Since the case $p=\infty$ is trivial, it suffices to consider $p\in[1;\infty)$. Suppose (\ref{eq6}) holds. For each $f\in L^p(\omega)$, since $s(t)\neq0$ almost everywhere, $\omega$ is homogeneous of order $\alpha$, and then applying Minkowski's inequality (see \cite[page 14]{hardy}) we shall obtain 
\begin{align*}
\|U_{\psi,s}f\|_{L^p(\omega)}=&\left(\int_{\mathbb R^n}\left|\int_0^1 f\left(s(t)\cdot x\right)\psi(t)dt\right|^p\omega(x)dx\right)^{1/p}\\
\leq& \int_0^1\left(\int_{\mathbb R^n}\left|f\left(s(t)\cdot x\right)\right|^p\omega(x)dx\right)^{1/p}\psi(t)dt\\
 =&\int_0^1\left(\int_{\mathbb R^n}\left|f\left(y\right)\right|^p\left|s(t)\right|^{-\alpha-n}\omega(y)dy\right)^{1/p}\psi(t)dt\\
=&\|f\|_{L^p(\omega)}\cdot\int_0^1|s(t)|^{-\frac{n+\alpha}p}\psi(t)dt<\infty.
\end{align*}
\vskip12pt
Thus, $U_{\psi,s}$ is defined as a bounded operator on $L^p(\omega)$ and the operator norm of $U_{\psi,s}$ on $L^p(\omega)$ is satisfied
\begin{equation}\label{th1eqa}
\|U_{\psi,s}\|_{L^p(\omega)\to L^p(\omega)}\leq \int_0^1|s(t)|^{-\frac{n+\alpha}p}\psi(t)dt.
\end{equation}
\vskip12pt 
Conversely, assuming that $U_{\psi,s}$ is defined as a bounded operator on $L^p(\omega)$. For any $0<\epsilon<1$, we put 
\begin{equation}\label{th1eqb}
f_\epsilon(x)=
\begin{cases}
0,\qquad\qquad\qquad\text{if $|x|\leq 1$},&\\
|x|^{-\frac{n+\alpha}p-\epsilon},\qquad\text{if $|x|>1$}.
\end{cases}
\end{equation}

Applying lemma \ref{lem1}, $f_\epsilon\in L^p(\omega)$ and $\|f_\epsilon\|_{L^p(\omega)}>0$, we have
\[
U_{\psi,s}f_\epsilon(x)=|x|^{-\frac{n+\alpha}p-\epsilon}\int_{S(t,x)}|s(t)|^{-\frac{n+\alpha}p-\epsilon}\psi(t)dt,
\]
here $S(t,x)=\{t\in[0,1]\,|\;\text{so that}\;|s(t)\cdot x|>1\}$. Hence
\[
\|U_{\psi,s}f_\epsilon\|^p_{L^p(\omega)}=\int_{\mathbb R^n}|x|^{-n-\alpha-p\epsilon}\left|\int_{S(t,x)}|s(t)|^{-\frac{n+\alpha}p-\epsilon}\psi(t)dt\right|^p\omega(x)dx.
\]
Since $|s(t)|\geq t^\beta$ for a.e. $t\in[0,1]$, there exists a measurable subset $E$ of $[0,1]$ with $|E|=0$ so that 
$$S(t,x)\supset\{t\in[0,1]\;|\;t\geq 1/|x|^{1/\beta}\}\setminus E.$$
Put $\delta=\epsilon^{-1}$, then
\[
\|U_{\psi,s}f_\epsilon\|^p_{L^p(\omega)}\geq \left(\int_{|x|\geq \delta^\beta}|x|^{-n-\alpha-p\epsilon}\omega(x)dx\right)\cdot\left(\int_{1/\delta}^1|s(t)|^{-\frac{n+\alpha}p-\epsilon}\psi(t)dt\right)^p
\]
\[
=\|f_\epsilon\|^p_{L^p(\omega)}\cdot \left(\delta^{-\beta\epsilon}\int_{1/\delta}^1|s(t)|^{-\frac{n+\alpha}p-\epsilon}\psi(t)dt\right)^p.
\]
So we have
\[
\|U_{\psi,s}\|_{L^p(\omega)\to L^p(\omega)}\geq\delta^{-\beta\epsilon}\int_{1/\delta}^1|s(t)|^{-\frac{n+\alpha}p-\epsilon}\psi(t)dt.
\]
Letting $\epsilon\to0^+$ we obtain
\begin{equation}\label{th1eqc}
\int_{0}^1|s(t)|^{-\frac{n+\alpha}p}\psi(t)dt\leq \|U_{\psi,s}\|_{L^p(\omega)\to L^p(\omega)}<\infty.
\end{equation}
From (\ref{th1eqa}) and (\ref{th1eqc}), we will receive (\ref{eq7}).

\end{proof}
\vskip12pt
Using the proof of theorem \ref{theo1}, we could find a sufficient condition on $\psi$ such that the integral operator $\mathcal U_{\psi,s}$, which is determined as
\[
\mathcal U_{\psi,s}f(x)=\int_0^\infty f\left(s(t)x\right)\psi(t)dt.
\]
is bounded on $L^p(\omega)$.
\begin{theorem}\label{theo1a} Let $p\in[1;\infty]$, $\alpha$ be real numbers and $\omega\in \mathcal W_{\alpha}$. Let $s:[0;1]\to\mathbb R$ be a measurable function. Then, $\mathcal U_{\psi,s}:L^p(\omega)\to L^p(\omega)$ exists as a bounded operator if, 
\begin{equation}\label{eq6c}
\int_0^\infty|s(t)|^{-\frac{n+\alpha}p}\psi(t)dt<\infty,
\end{equation}
and 
\[
\|\mathcal U_{\psi,s}\|_{L^p(\omega)\to L^p(\omega)}\leq \int_0^\infty|s(t)|^{-\frac{n+\alpha}p}\psi(t)dt<\infty.
\]
\end{theorem}
\vskip12pt
Using theorem \ref{theo1}, we could deduce some generalizations of Hardy's integral inequality. In case $s(t)=t$, $\omega\equiv1$ we obtain the above J. Xiao's result (see \cite[page 662]{xiao}). On the other hand, in case $n=1$, $\omega(x)=|x|^{p-b-1}$, and $\psi(t)\equiv1$ we get 
\begin{equation}\label{th1eqd}
\left(\int_0^\infty\left(\int_0^x|f(t)|dt\right)^px^{-b-1}dx\right)^{1/p}\leq \frac pb\left(\int_0^\infty |f(t)|^pt^{p-b-1}dt\right)^{1/p},
\end{equation}
and in case $n=1$, $s(t)=\frac1t$, $\psi(t)=t^{-2}$ and $\omega(x)=|x|^{p+b-1}$, then
\begin{equation}\label{th1eqe}
\left(\int_0^\infty\left(\int_0^x|f(t)|dt\right)^px^{-b-1}dx\right)^{1/p}\leq \frac pb\left(\int_0^\infty |f(t)|^pt^{p-b-1}dt\right)^{1/p}.
\end{equation}
Two such elementary well-known Hardy integral inequalities can be found in \cite[page 29]{grafakos}. 
\begin{corollary}
\label{coro1} Let $p\in[1;\infty]$, $\alpha$ be real numbers and $\omega\in \mathcal W_{\alpha}$. Let $s:[0;1]\to\mathbb R$ be a measurable function so that $|s(t)|\geq t^\beta$ a.e. $t\in[0,1]$, for some constant $\beta>0$. Then, $V_{\psi,s}:L^p(\omega)\to L^p(\omega)$ exists as a bounded operator if and only if 
\begin{equation}\label{eq8}
\int_0^1|s(t)|^{n-\frac{n+\alpha}p}\psi(t)dt<\infty.
\end{equation}
Moreover, when (\ref{eq8}) holds, the operator norm of $V_{\psi,s}$ on $L^p(\omega)$ is given by
\begin{equation}\label{eq9}
\|V_{\psi,s}\|_{L^p(\omega)\to L^p(\omega)}=\int_0^1|s(t)|^{n-\frac{n+\alpha}p}\psi(t)dt.
\end{equation}
\end{corollary}
\begin{proof} This is an immediate consequence of theorem \ref{theo1} with the relation $V_{\psi,s}f(x)=U_{|s(\cdot)|^n\psi,s}f(x)$.
\end{proof}

\begin{corollary}Let $s,\psi$ be measurable functions  on $[0,1]$, and $\omega\in\mathcal W_\alpha$, and \vspace{8pt}
\begin{itemize}
\item[{(i)}] $\int\limits_0^1|s(t)|^{-\frac {n+\alpha }p}\psi(t)dt<\infty$,\vspace{8pt}

\item[{(ii)}] There are two positive real numbers $\beta,\gamma$ so that $t^\beta\leq |s(t)|\leq t^{-\gamma}$ a.e. $t\in[0,1]$.\vspace{8pt}
\end{itemize}
Then two operators $U_{\psi,s}$ and $V_{|s(\cdot)|^{-\alpha}\psi,1/s}$ are mutually adjoint in the sense: for any $f\in L^p(\omega)$, $g\in L^q(\omega)$, $1<p<\infty$ and $\frac 1p+\frac1q=1$, we have 
\begin{equation}\label{eq9a}
\int_{\mathbb R^n}g(x)U_{\psi,s}f(x)\omega(x)dx=\int_{\mathbb R^n}f(y)\left(V_{|s(\cdot)|^{-\alpha}\psi,1/s}g(y)\right)\omega(y)dy.
\end{equation}
\end{corollary}
\begin{proof}
If $f\in L^p(\omega)$, $g\in L^q(\omega)$, theorem \ref{theo1} and corollary \ref{coro1} show that $U_{\psi,s}f\in L^p(\omega)$ and $V_{|s(\cdot)|^{-\alpha}\psi,1/s}g\in L^q(\omega)$. Hence, both sides of (\ref{eq9a}) are finite. With the help of Fubini theorem, we have
\begin{align*}
\int_{\mathbb R^n}g(x)U_{\psi,s}f(x)\omega(x)dx=&\int_{\mathbb R^n}g(x)\left(\int_0^1f\left(s(t)x\right)\psi(t)dt\right)\omega(x)dx\\ 
=&\int_0^1\left(\int_{\mathbb R^n}g(x)f\left(s(t)x\right)\omega(x)dx\right)\psi(t)dt\\
=&\int_0^1\left(\int_{\mathbb R^n}g\left(y/s(t)\right)f\left(y\right)\omega(y)dx\right)|s(t)|^{-n-\alpha}\psi(t)dt\\	
=&\int_{\mathbb R^n}f\left(y\right)\omega(y)\left(\int_0^1g\left(y/s(t)\right)|s(t)|^{-n}\cdot|s(t)|^{-\alpha}\psi(t)dt\right)dy\\	
=&\int_{\mathbb R^n}f(y)\left(V_{|s(\cdot)|^{-n}\psi,1/s}g(y)\right)\omega(y)dy.
\end{align*}
\end{proof}

Now we characterize the weight function $\psi$ for which $U_{\psi,s}$ is bounded on $BMO(\omega)$, and we calculate the $BMO(\omega)-$norm of $U_{\psi,s}$.
\begin{theorem}\label{theo2}Let $p\in[1;\infty]$ be real number and $\omega$ belongs to $\mathcal W=\bigcup\limits_{\alpha>-n}\mathcal W_{\alpha}$. Let $s:[0;1]\to\mathbb R$ be a measurable function so that $s(t)\neq0$ almost everywhere on $[0,1]$.
\begin{itemize}
\item[{(i)}] If \begin{equation}\label{eq10}
\int_0^1\psi(t)dt<\infty,
\end{equation}
then $U_{\psi,s}:BMO(\omega)\to BMO(\omega)$ exists as a bounded operator and
\[
\|U_{\psi,s}\|_{BMO(\omega)\to BMO(\omega)}\leq\int_0^1\psi(t)dt.
\]
\item[{(ii)}] If $n=1$ and $U_{\psi,s}:BMO(\omega)\to BMO(\omega)$ exists as a bounded operator, then
\begin{equation}\label{eq11}
\left|\int_0^1{\rm sgn}s(t)\cdot \psi(t)dt\right|<\infty.
\end{equation}
Moreover, if $s(t)$ does not change sign on $[0,1]$, then the operator norm of $U_{\psi,s}$ on $BMO(\omega)$ is given by
\begin{equation}\label{eq12}
\|U_{\psi,s}\|_{BMO(\omega)\to BMO(\omega)}=\int_0^1\psi(t)dt.
\end{equation}
\item[{(iii)}] If $n>1$, then  $U_{\psi,s}:BMO(\omega)\to BMO(\omega)$ exists as a bounded operator, if and only if (\ref{eq10}) holds. Moreover, the norm of $U_{\psi,s}$ on $L^p(\omega)$ is also given by
\[
\|U_{\psi,s}\|_{BMO(\omega)\to BMO(\omega)}=\int_0^1\psi(t)dt.
\]
\end{itemize}
\end{theorem}

\begin{proof}

(i) Suppose (\ref{eq10}) holds and $\omega\in\mathcal W_\alpha$ for some $\alpha\geq-n$. Let $f$ be in $BMO(\omega)$, and $B$ be any ball of $\mathbb R^n$. By using Fubini's theorem, we have
\[
\left(U_{\psi,s}f\right)_{B,\omega}=\frac1{\omega(B)}\int_B\left(\int_0^1f\left(s(t)\cdot x\right)\psi(t)dt\right)\omega(x)dx= 
\]
\[
\int_0^1\left(\frac1{\omega(B)}\int_{s(t)\cdot B}f(y)\cdot|s(t)|^{-n-\alpha}\omega(y)dy\right)\psi(t)dt=\int_0^1f_{s(t)\cdot B,\omega}\psi(t)dt.
\]
Thus
\begin{align*}
&\qquad\frac1{\omega(B)}\int_B\left|U_{\psi,s}f(x)-\left(U_{\psi,s}f\right)_{B,\omega}\right|\omega(x)dx\\
&\leq\frac1{\omega(B)}\int_B\left(\int_0^1\left|f\left(s(t)\cdot x\right)-f_{s(t)\cdot B,\omega}\right|\psi(t)dt\right)\omega(x)dx\\
&=\int_0^1\left(\frac1{\omega(B)}\int_B\left|f\left(s(t)\cdot x\right)-f_{s(t)\cdot B,\omega}\right|\omega(x)dx\right)\psi(t)dt\\
&=\int_0^1\left(\frac1{\omega(B)}\int_{s(t)\cdot B}\left|f\left(y\right)-f_{s(t)\cdot B,\omega}\right|\omega(y)|s(t)|^{-n-\alpha}dy\right)\psi(t)dt\\
&=\int_0^1\left(\frac1{\omega(s(t)\cdot B)}\int_{s(t)\cdot B}\left|f\left(y\right)-f_{s(t)\cdot B,\omega}\right|\omega(y)dy\right)\psi(t)dt\\
&\leq \|f\|_{BMO(\omega)}\cdot\int_0^1\psi(t)dt.
\end{align*}
Hence, $\|U_{\psi,s}f\|_{BMO(\omega)}\leq  \|f\|_{BMO(\omega)}\cdot\int_0^1\psi(t)dt$, so $U_{\psi,s}$ is bounded on $BMO(\omega)$, and in this case
\begin{equation}\label{eq13}
\|U_{\psi,s}\|_{BMO(\omega)\to BMO(\omega)}\leq\int_0^1\psi(t)dt.
\end{equation}
\vskip12pt
(ii) We assume that $n=1$ and $U_{\psi,s}$ is bounded on $BMO(\omega)$. Let $f_0(x_1)={\rm sgn} x_1$, then for any segment $B=(x_0-r,x_0+r)$, where $r>0$, then
\[
\left(f_0\right)_{B,\omega}=\frac{\omega(B_+)-\omega(B_{-})}{\omega(B_+)+\omega(B_{-})},
\]
where $B_+=\{x\in B:\;x\geq0\}$ and $B_{-}=\{x\in B:\;x<0\}$. Using a well-known inequality $\frac{2ab}{(a+b)^2}\leq \frac12$ for any $a,b\geq0$, we shall obtain
\[
\frac1{\omega(B)}\int_B\left|f_0(x)-\left(f_0\right)_{B,\omega}\right|\omega(x)dx=\frac{2\omega(B_+)\cdot\omega(B_{-})}{\omega(B)^2}\leq \frac12.
\]
So we have proved that $f_0\in BMO(\omega)$, and $0<\|f_0\|_{BMO(\omega)}\leq\frac12$.
\vskip12pt
Since $f_0(tx)={\rm sgn}\, t\cdot f_0(x)$ for all $t\neq0$, we have
\[
U_{\psi,s}f_0(x)=f_0(x)\int_0^1{\rm sgn}\,s(t)\cdot \psi(t)dt.
\]
Thus
\[
\|U_{\psi,s}f_0\|_{BMO(\omega)}=\|f_0\|_{BMO(\omega)}\cdot\left|\int_0^1{\rm sgn}s(t)\cdot\psi(t)dt\right|.
\]
This implies 
\begin{equation}\label{eq14}
\left|\int_0^1{\rm sgn}s(t)\cdot\psi(t)dt\right|\leq\|U_{\psi,s}\|_{BMO(\omega)\to BMO(\omega)}<\infty.
\end{equation}

\vskip12pt
(iii) We consider the case $n>1$. In (i) we have proved that, the boundedness of $U_{\psi,s}$ could be implied from (\ref{eq10}). So now we assume that $U_{\psi,s}$ is bounded on $L^p(\omega)$. Let $S_{n}=\{x\in\mathbb{R}^n:|x|=1\}$ be the unit sphere in $\mathbb{R}^n$ and $\phi\colon S_{n}\to\mathbb{R}$ an essential upper bounded, even, non-constant function, that is
\[
\text{\rm ess.sup}_{t\in S_{n}}|\phi(t)|=\|\phi\|_\infty<\infty,\quad\phi(-t)=\phi(t)\qquad(\forall\; t\in S_{n}).
\]
Define
$$
f_1(x)=\begin{cases}
\phi(x/|x|) \quad\text{if}\; x\neq0,&\\
\quad 0 \;\;\;\;\qquad\text{if }\,x=0.
\end{cases}
$$
Since $f_1\in L^\infty(\mathbb R^n)$, then $f_1\in BMO(\omega)$ and $\|f_1\|_{BMO(\omega)}\neq0$. Note that $f_1(tx)=f_1(x)$ for $t\neq 0$, so 
\[
U_{\psi,s}f_1(x)=f_1(x)\int_0^1\psi(t)dt.
\]
Taking $BMO(\omega)-$norm both sides we have
\[
\|f_1\|_{BMO(\omega)}\int_0^1\psi(t)dt=\|U_{\psi,s}f_1\|_{BMO(\omega)}\leq \|f_1\|_{BMO(\omega)}\cdot \|U_{\psi,s}\|_{BMO(\omega)\to BMO(\omega)}.
\]
Thus (\ref{eq10}) holds and we have 
\[
\|U_{\psi,s}\|_{BMO(\omega)\to BMO(\omega)}= \int_0^1\psi(t)dt.
\]
\end{proof}

\vskip12pt
In theorem \ref{theo2}, let us put $s(t)=t$ and $\omega\equiv1$, we shall get the result of J. Xiao for $U_\psi$ as follows
\begin{corollary} $U_\psi: \,BMO(\mathbb R^n)\to BMO(\mathbb R^n)$ exists as a bounded operator, if and only if
\[
\int_0^1\psi(t)dt<\infty.
\]
Moreover, the norm operator of $U_\psi$ on $BMO(\mathbb R^n)$ is given by
\[
\|U_{\psi,s}\|_{BMO(\omega)\to BMO(\omega)}= \int_0^1\psi(t)dt.
\]
\end{corollary}
\section{Commutator of the generalized Hardy-Ces\`{a}ro operator}
Recently, Zun Wei Fu, Zong Guang Liu, and Shan Zhen Lu \cite{fu1} established a sufficient and necessary condition on weight function $\psi(t)$, which ensures the boundedness of the commutators (with symbols in $BMO(\mathbb R^n)$) of weighted Hardy operators $U_\psi$ and weighted Ces\`{a}ro operator $V_\psi$ on $L^p(\mathbb R^n)$. The purpose of this section is to extend their results to generalized Hardy-Ces\`{a}ro operator $U_{\psi,s}$ on weighted $L^p(\omega)$ spaces.  
\vskip12pt
Let $b$ be a measurable, locally integrable function and $\psi:[0,1]\to[0,\infty)$ be a function. We define the commutator of weighted Hardy-Ces\`{a}ro operator  $U_{\psi,s}$ as
\[
U^b_{\psi,s}f:=bU_{\psi,s}f-U_{\psi,s}(bf).
\]
Similarly we define the commutator of weighted Ces\`{a}ro operator $V^b_{\psi,s}$ as
\[
V^b_{\psi,s}f:=bV_{\psi,s}(bf)-bV_{\psi,s}(f).
\]
\vskip12pt
It is known that commutators of (singular) integral transforms have been of interest in many contexts, particularly in the theory of P.D.E. The Hardy operator $U$ is well-known to be controlled by Hardy-Littlewood maximal operator $M$. To obtain estimates for commutator of generalized Hardy-Ces\`{a}ro operator $U_{\psi,s}$, we use the weight Hardy-Littlewood maximal operator $M_\omega$ to control $U_{\psi,s}$. Our main result in this section is formulated as follows.
\begin{theorem}\label{theo3}
Let $\psi:[0,1]\to[0,\infty)$, $s:[0,1]\to\mathbb R$ be measurable functions, $\omega\in\mathcal W_\alpha$ has doubling property, $\alpha>-n$ and $1<p<\infty$. We also assume that there exists real constants $\beta,\gamma$ such that $\beta\gamma>0$ and $t^\beta\leq |s(t)|\leq t^\gamma$ for almost everywhere $t\in[0,1]$. 
\begin{itemize}
\item[{(a)}] If $U^b_{\psi,s}$ is bounded on $L^p(\omega)$ for any $b\in BMO(\omega)$, then
\begin{equation}\label{eq15}
\int\limits_0^1|s(t)|^{-\frac{n+\alpha}p}\cdot\left|\log\frac1{ |s(t)|}\right|\cdot\psi(t)dt<\infty.
\end{equation}
Moreover, if $\beta,\gamma>0$, then $\int\limits_0^{1/\delta} \psi(t)dt<\frac C{\left(\delta^{\gamma(\alpha+n)}-1\right)^{1/p}}$,
where $C$ is a positive constant, which does not depend on $\delta$.
\item[{(b)}] If $|s(t)|=t^\gamma$, and 
\begin{equation}\label{eq15a}
\begin{cases}
\int_0^1|s(t)|^{-\frac{n+\alpha}p}\cdot\left|\log\frac1{ |s(t)|}\right|\cdot\psi(t)dt<\infty,\vspace*{12pt}&\\
\int_0^1|s(t)|^{-\frac{n+\alpha}p}\cdot\psi(t)dt<\infty,&\\
\end{cases}
\end{equation}
then $U^b_{\psi,s}$ is bounded on $L^p(\omega)$ for any $b\in BMO(\omega)$ with operator norm
\[
\|U^b_{\psi,s}\|_{L^p(\omega)}\leq C_\omega \left(\int_0^1|s(t)|^{-\frac{n+\alpha}p}\psi(t)\left(2+\left|\log\frac1{ |s(t)|}\right|\right)dt\right)\cdot\|b\|_{BMO(\omega)}.
\]
\end{itemize}
\end{theorem}
\vskip12pt\noindent
{\bf Remark.} We observe that, (\ref{eq15}) implies $\int\limits_0^{1/\delta}|s(t)|^{-\frac{n+\alpha}p}\cdot\psi(t)dt<\infty$, for any $\delta>1$. But unfortunately, as is shown in \cite{fu1}, (\ref{eq15}) is not enough to imply $$\int_0^1|s(t)|^{-\frac{n+\alpha}p}\cdot\psi(t)dt<\infty.$$

\vskip12pt\noindent
To prove theorem \ref{theo3}, we need the following lemma
\begin{lemma}\label{lem4}
If $\omega\in\mathcal W_\alpha$ and has doubling property, then $\log|x|\in BMO(\omega)$.
\end{lemma}
Although lemma \ref{lem4} seems to be well-known, but we could not find any proof of it in the literature. One way to prove it is as follows.
\begin{proof} To prove $\log|x|\in BMO(\omega)$, for any $x_0\in\mathbb R^n$ and $r>0$, we must find a constant $c_{x_0,r}$, such that $\frac1{\omega\left(B(x_0,r)\right)}\int_{|x-x_0|\leq r}\left|\log|x|-c_{x_0,r}\right|\omega(x)dx$ is uniformly bounded. Since  
\[
\frac1{\omega\left(B(x_0,r)\right)}\int\limits_{|x-x_0|\leq r}\left|\log|x|-c_{x_0,r}\right|\omega(x)dx\]
\[
=\frac{r^{\alpha+n}}{\omega\left(B(x_0,r)\right)}\int\limits_{|z-r^{-1}x_0|\leq1}\left|\log|z|-\log r-c_{x_0,r}\right|\omega(z)dz
\]
\[
=\frac1{\omega\left(B(r^{-1}x_0,1)\right)}\int\limits_{|z-r^{-1}x_0|\leq1}\left|\log|z|-\log r-c_{x_0,r}\right|\omega(z)dz,
\]
we may take $c_{x_0,r}=c_{r^{-1}x_0,1}+\log r$, and so things reduce to the case that $r=1$ and $x_0$ is arbitrary. Let 
\[
A_{x_0}=\frac1{\omega\left(B(x_0,1)\right)}\int\limits_{|z-x_0|\leq1}\left|\log|z|-c_{x_0,1}\right|\omega(z)dz.
\]
If $|x_0|\leq2$, we take $c_{x_0,1}=0$, and observe that
\[
A_{x_0}\leq \frac1{\omega\left(B(x_0,1)\right)}\int\limits_{|z|\leq3}\log3\cdot\omega(z)dz= \log3\cdot\frac{\omega\left(B(0,3)\right)}{\omega\left(B(x_0,1)\right)}
\]
\[
\leq \log3\cdot \frac{\omega\left(B(x_0,6)\right)}{\omega\left(B(x_0,1)\right)}\leq C<\infty,
\]
where the last inequality comes from the assumption that $\omega$ has doubling property.
\vskip12pt
If $|x_0|\geq2$, take $c_{x_0,1}=\log |x_0|$. In this case, notice that
\[
A_{x_0}=  \frac1{\omega\left(B(x_0,1)\right)}\int\limits_{B(x_0,1)}\left|\log\frac{|z|}{|x_0|}\right|\omega(z)dz
\]
\[
\leq \frac1{\omega\left(B(x_0,1)\right)}\int\limits_{B(x_0,1)}\max\left\{\log\frac{|x_0|+1}{|x_0|}\,,\,\log\frac{|x_0|}{|x_0|-1}\right\}\cdot\omega(z)dz
\]
\[
\leq \max\left\{\log\frac{|x_0|+1}{|x_0|}\,,\,\log\frac{|x_0|}{|x_0|-1}\right\}\leq\log2.
\]
Thus $\log|x|$ is in $BMO(\omega)$.
\end{proof} 
\vskip12pt\noindent
{\it Proof of theorem \ref{theo3}}
{(a)} To complete the proof of (a) we need to consider two cases when $\beta,\gamma$ are both positive or both negative. Let us consider the first case when $\beta,\gamma>0$. We assume that $\|U^b_{\psi,s}\|_{L^p(\omega)\to L^p(\omega)}<\infty$ for any $b\in BMO$. Set $b(x)=\log |x|\in BMO(\omega)$ (by lemma \ref{lem4}). For any $0<\epsilon<1$, take $f_\epsilon$ as in (\ref{th1eqb}), then
\[
U^b_{\psi,s}f_\epsilon(x)=-|x|^{-\frac{n+\alpha}p-\epsilon}\int_{S(t,x)}|s(t)|^{-\frac{n+\alpha}p-\epsilon}\log|s(t)|\cdot\psi(t)dt,
\]
here $S(t,x)=\{t\in[0,1]\,|\;\text{so that}\;|s(t)\cdot x|>1\}$. Putting $\delta=\epsilon^{-1}>1$, note that $\log|s(t)|$ does not change sign almost everywhere on $[0,1]$, so
\[
\|U^b_{\psi,s}f_\epsilon\|^p_{L^p(\omega)}=\int_{\mathbb R^n}|x|^{-n-\alpha-p\epsilon}\omega(x)\left(\int_{S(t,x)}|s(t)|^{-\frac{n+\alpha}p-\epsilon}\cdot\log\frac1{|s(t)|}\psi(t)dt\right)^pdx
\]
\[
\geq\int_{\mathbb R^n}|x|^{-n-\alpha-p\epsilon}\omega(x)\left(\int_{\min\{1/|x|^{1/\beta},1\}}^1|s(t)|^{-\frac{n+\alpha}p-\epsilon}\cdot\log\frac1{|s(t)|}\psi(t)dt\right)^pdx
\]
\[
\geq\left(\int_{|x|> \delta^\beta}|x|^{-(n+\alpha+p\epsilon)}\omega(x)dx\right)\cdot\left(\int_{1/\delta}^1|s(t)|^{-\frac{n+\alpha}p-\epsilon}\cdot\log\frac1{|s(t)|}\cdot\psi(t)dt\right)^p
\]
\[
=\left(\int_{|y|>1}|y|^{-(n+\alpha+p\epsilon)}\omega(y)dy\right)\cdot\left(\frac1{\delta^\epsilon}\int_{1/\delta}^1|s(t)|^{-\frac{n+\alpha}p-\epsilon}\cdot\log\frac1{|s(t)|}\cdot\psi(t)dt\right)
\]
\[
=\|f_\epsilon\|^p_{L^p(\omega)}\cdot \left(\epsilon^\epsilon\int_{1/\delta}^1|s(t)|^{-\frac{n+\alpha}p-\epsilon}\cdot\log\frac1{|s(t)|}\cdot\psi(t)dt\right)^p.
\]
As in the proof of theorem \ref{theo1}, $\|f_\epsilon\|^p_{L^p(\omega)}>0$, so 
\[
\epsilon^\epsilon\int_{1/\delta}^1|s(t)|^{-\frac{n+\alpha}p-\epsilon}\cdot\log\frac1{|s(t)|}\cdot\psi(t)dt\leq \|U^b_{\psi,s}\|_{L^p(\omega)\to L^p(\omega)}<\infty.
\]
Letting $\epsilon\to0^+$, then 
\[
\int_{0}^1|s(t)|^{-\frac{n+\alpha}p}\cdot\log\frac1{|s(t)|}\cdot\psi(t)dt<\infty.
\]
Now we set $b(x)=f(x)=1_{B(0,1)}(x)$, the characteristic function of the ball $B(0,1)$. Note that  $f\in L^p(\omega)$, since $\alpha>-n$ and $b\in BMO(\mathbb R^n)$. So we have
\[
U^b_{\psi,s}g_\epsilon(x)=\int_0^1\left(b(x)-b\left(s(t)\cdot x\right)\right)f\left(s(t)\cdot x\right)\psi(t)dt
\]
\[
=\left(1_{B(0,1)}(x)-1\right)\cdot\int_{\overline{S}(t,x)} \psi(t)dt,
\]
where $\overline{S}(t,x)=\{t\in[0,1]:\;|s(t)\cdot x|<1\}$. Since  $|s(t)|\leq t^\gamma$ for a.e. on $[0,1]$, we can find a zero measurable $E\subset [0,1]$ so that  $\overline{S}(t,x)\supset \{t\in[0,1]:\;t<\frac1{|x|^{1/\gamma}}\}\setminus E$. For any $\delta>1$, we obtain
\[
\|U^b_{\psi,s}f\|^p_{L^p(\omega)}\geq\int_{1<|x|<\delta^\gamma}\omega(x)\cdot\left( \int_0^{1/|x|^{1/\gamma}}\psi(t)dt\right)^pdx
\]
\[
\geq\left(\int_{1<|x|< \delta^\gamma} \omega(x)dx\right)\cdot \left(\int_0^{1/\delta}\psi(t)dt\right)^p
\]
\[
=\|f\|^p_{L^p(\omega)}\cdot\left(\delta^{\gamma(\alpha+n)}-1\right)\cdot\left(\int_0^{1/\delta}\psi(t)dt\right)^p.
\]
Therefore, for any $\delta>1$ we get
\[
\int_0^{1/\delta}\psi(t)dt\leq \frac C{\left(\delta^{\gamma(\alpha+n)}-1\right)^{1/p}}<\infty,
\]
where $C=\|U^b_{\psi,s}\|_{L^p(\omega)}$. 
\vskip12pt
The left case when $\beta,\gamma<0$, can be proved in the same manner as shown, with $g_\epsilon$ used in place of $f_\epsilon$, where
\[
g_\epsilon(x)=
\begin{cases}
0\qquad\qquad\qquad\text{if}\;|x|\geq1,&\\
|x|^{-\frac{n+\alpha}p+\epsilon}\,\qquad\text{if}\;|x|<1.
\end{cases}
\]
\vskip12pt
(b) We assume that $|s(t)|=t^\gamma$, $\omega$ has doubling property, and (\ref{eq15}) holds. Let $b$ be any function in $BMO(\omega)$. We shall show that, $U_{\psi,s}$ can be determined as a bound operator on $L^p(\omega)$. For any ball $B$ of $\mathbb R^n$, $x\in B$, we have
\begin{align*}
&\frac1{\omega(B)}\int_B\left|U^b_{\psi,s}f(y)\right|\omega(y)dy\\
&\leq\frac1{\omega(B)}\int_B\int_0^1\left|\left(b(y)-b\left(s(t)\cdot y\right)\right)f\left(s(t)\cdot y\right)\right|\psi(t)\omega(y)dtdy\\
&=\frac1{\omega(B)}\int_0^1\int_B\left|\left(b(y)-b\left(s(t)\cdot y\right)\right)f\left(s(t)\cdot y\right)\right|\omega(y)dy\psi(t)dt\\
&\leq \frac1{\omega(B)}\int_0^1\int_B\left|\left(b(y)-b_{B,\omega}\right)f\left(s(t)\cdot y\right)\right|\omega(y)dy\psi(t)dt\\ 
&\quad+\frac1{\omega(B)}\int_0^1\int_B\left|\left(b_{B,\omega}-b_{s(t)B,\omega}\right)f\left(s(t)\cdot y\right)\right|\omega(y)dy\psi(t)dt\\
&\quad+\frac1{\omega(B)}\int_0^1\int_B\left|\left(b(s(t)y)-b_{s(t)B,\omega}\right)f\left(s(t)\cdot y\right)\right|\omega(y)dy\psi(t)dt\\
&=:I_1+I_2+I_3.
\end{align*}
Now take any $1<r<p$. 
\vskip12pt 
{\it Estimate for $I_1$:} By H\"{o}lder's inequality $(1/r+1/r^\prime=1$), we deduce
\[
I_1\leq\int_0^1\left(\frac1{\omega(B)}\int_B\left|f\left(s(t)y\right)\right|^r\omega(y)dy\right)^{1/r}\cdot\left(\frac1{\omega(B)}\int_B\left|b(y)-b_{B,\omega}\right|^{r^\prime}\omega(y)dy\right)^{1/r^\prime}\psi(t)dt.
\]
\[
\leq C_\omega\cdot \|b\|_{BMO(\omega)}\cdot \int_0^1\left(\frac1{\omega(s(t)B)}\int_{s(t)B}|f(y)|^r\omega(y)dy\right)^{1/r}\psi(t)dt
\]
\[
\leq C_\omega\cdot\|b\|_{BMO(\omega)}\cdot\int_0^1\left(M_\omega f^r\left(s(t)x\right)\right)^{1/r}\psi(t)dt,
\]
where $C_\omega$ is a constant, which depends only on $\omega$.
\vskip12pt 
{\it Estimate for $I_3$:} Applying H\"{o}lder's inequality again, we thus obtain
\[
I_3\leq\int_0^1\left(\frac1{\omega(B)}\int_B\left|f\left(s(t)y\right)\right|^r\omega(y)dy\right)^{1/r}\cdot\left(\frac1{\omega(B)}\int_B\left|b\left(s(t)y\right)-b_{s(t)B,\omega}\right|^{r^\prime}\omega(y)dy\right)^{1/r^\prime}\psi(t)dt.
\]
\[
=\int_0^1\left(\frac1{\omega(s(t)B)}\int_{s(t)B}|f(y)|^r\omega(y)dy\right)^{1/r}\cdot\left(\frac1{\omega(s(t)B)}\int_{s(t)B}\left|b\left(y\right)-b_{s(t)B,\omega}\right|^{r^\prime}\omega(y)dy\right)^{1/r^\prime}\psi(t)dt
\]
\[
\leq C_\omega\cdot\|b\|_{BMO(\omega)}\cdot\int_0^1\left(M_\omega f^r\left(s(t)x\right)\right)^{1/r}\psi(t)dt.
\]
\vskip12pt 
{\it Estimate for $I_2$:} We need the following lemmas. The first one generalizes a lemma of Torre and Torrea (see lemma 1.10 in \cite{torre}).
\begin{lemma}\label{lem2}
Let $\omega$ be a doubling weight function. Then, there exists a positve constant $C$ so that for any balls $B_1=B(x_1,r_1)$, $B_2=B(x_2,r_2)$, whose intersection is not empty,  and $\frac12r_2\leq r_1\leq 2r_2$, then $\omega(B)\leq C\omega(B_i)$, $i=1,2$. Here, $B$ is the smallest ball, which contains both $B_1$ and $B_2$. Moreover, for each function $b\in BMO(\omega)$, we have 
\[
\left|b_{B_1,\omega}-b_{B_2,\omega}\right|\leq 2C\|b\|_{BMO(\omega)}.
\]
\end{lemma}
\begin{proof}Since $\omega$ has doubling property, there exists a constant $C_1$, so that $\omega\left(B(x,2r)\right)\leq C_1\omega\left(B(x,r)\right)$ for any $x\in\mathbb R^n$ and $r>0$. Without lost of generality, we assume $r_2\leq r_1\leq 2r_2$. Let $B_1=B(x_1,r_1)$, $B_2=B(x_2,r_2)$ be two balls, whose intersection is not empty and $r_1\geq r_2$. Take $x\in B_1\cap B_2$. Then,
\[
\omega(B)\leq \omega\left(B(x,2r_1\right)\leq C_1\omega\left(B(x,r_1)\right)\leq C_1\omega\left(B(x_1,2r_1)\right)\leq C_1^2\omega\left(B_1\right),
\]
and
\[
\omega(B)\leq \omega\left(B(x,4r_2\right)\leq C_1^2 \omega\left(B(x,r_2)\right)\leq C_1^3\omega\left(B(x_2,r_2)\right).
\]
Thus we can choose the constant $C=\max\{C_1^2,C_1^3\}$.
\vskip12pt
It is clear that 
\[
\left|b_{B_1,\omega}-b_{B_2,\omega}\right|\leq \left|b_{B_1,\omega}-b_{B,\omega}\right|+\left|b_{B,\omega}-b_{B_2,\omega}\right|.
\]
Now
\[
\left|b_{B,\omega}-b_{B_1,\omega}\right|=\left|b_{B,\omega}-\frac1{\omega(B_1)}\int_{B_1}b(y)\omega(y)dy\right|
\]
\[
\leq\frac1{\omega(B_1)}\int_{B_1}\left|b(y)-b_{B,\omega}\right|\omega(y)dy\leq\frac C{\omega(B)}\int_B\left|b(y)-b_{B,\omega}\right|\omega(y)dy\leq C\|b\|_{BMO(\omega)}.
\]
The left term is estimated in a similar way.
\end{proof}
\begin{lemma}\label{lem3} Let $B=B(x_0,r)$ and $\gamma\neq0$. There exists $a\in(1,2]$, such that
\begin{itemize}
\item[{(a)}] $a^{-i}B\cap a^{-(i+1)}B\neq\emptyset$ for any $i\in\mathbb N\cup\{0\}$, \vspace*{8pt}
\item[{(b)}] if $t\in\left[a^{-(i+1)/\gamma},a^{-i/\gamma}\right]$ then $a^{-(i+1)}B\cap t^\gamma B\neq\emptyset$. Here $\left[a^{-(i+1)/\gamma},a^{-i/\gamma}\right]$ is the set of all real numbers between $a^{-(i+1)/\gamma}$ and $a^{-i/\gamma}$.
\end{itemize}
\end{lemma}
\begin{proof}Let us consider the case when $\gamma>0$, the left case when $\gamma<0$ is similar. We need to choose $a\in(1,2]$ so that $\left(a^{-i}+a^{-i-1}\right)r\geq |x_0|\left(a^{-i}-a^{-i-1}\right)$ and $\left(t^\gamma+a^{-i-1}\right)r\geq |x_0|\left(t^\gamma-a^{-i-1}\right)$ for any $t\in\left[a^{-(i+1)/\gamma},a^{-i/\gamma}\right]$. These are equivalent to
\[
\begin{cases}
(a+1)\geq \frac{|x_0|}r(a-1)\vspace*{8pt}&\\
t^\gamma\left(1-\frac{|x_0|}r\right)+a^{-i-1}\left(1+\frac{|x_0|}r\right)\geq0\quad\text{for any }\;t\in\left[a^{-(i+1)/\gamma},a^{-i/\gamma}\right].
\end{cases}
\]
But this is easy to choose by letting $a\to 1^+$.

\end{proof}
Now we turn to estimate $I_2$. Let $a$ be as in lemma \ref{lem3}. H\"{o}lder's inequality implies
\[
I_2=\frac1{\omega(B)}\int_0^1\int_B\left|\left(b_{B,\omega}-b_{s(t)B,\omega}\right)f\left(s(t)\cdot y\right)\right|\omega(y)dy\psi(t)dt
\]
\[
=\int_0^1\left(\frac1{\omega(B)}\int_B\left|f\left(s(t)\cdot y\right)\right|\omega(y)dy\right)\left|b_{B,\omega}-b_{s(t)B,\omega}\right|\psi(t)dt
\]
\[
\leq\int_0^1\left(M_\omega f^r(s(t)x)\right)^{1/r}\left|b_{B,\omega}-b_{s(t)B,\omega}\right|\psi(t)dt
\]
\[
=\sum_{k=0}^\infty\int_{a^{-(k+1)/\gamma}}^{a^{-k/\gamma}}\left(M_\omega f^r(s(t)x)\right)^{1/r}\left|b_{B,\omega}-b_{s(t)B,\omega}\right|\psi(t)dt
\]
\[
\leq\sum_{k=0}^\infty\int_{a^{-(k+1)/\gamma}}^{a^{- k/\gamma}}\left(M_\omega f^r(s(t)x)\right)^{1/r}\left(\sum_{i=0}^k\left|b_{a^{-(i+1)}B,\omega}-b_{a^{- i}B,\omega}\right|+\left|b_{a^{-(k+1)}B,\omega}-b_{s(t)B,\omega}\right|\right)\psi(t)dt.
\]

Applying lemma \ref{lem2} and lemma \ref{lem3}, we have $\left|b_{a^{-(i+1)}B,\omega}-b_{a^{- i}B,\omega}\right|\leq C\|b\|_{BMO(\omega)}$, and $\left|b_{a^{-(k+1)}B,\omega}-b_{s(t)B,\omega}\right|\leq C\|b\|_{BMO(\omega)}$. Also, from $t\in[a^{-(k+1)/\gamma},a^{-k/\gamma}]$, it implies that $k\leq C\log a^{k}\leq C\left|\log\frac1{ |s(t)|}\right|$. From these, we have
\[
I_2\leq C\|b\|_{BMO(\omega)}\sum_{k=0}^\infty\int_{a^{-(k+1)/\gamma}}^{a^{- k/\gamma}}\left(M_\omega f^r(s(t)x)\right)^{1/r}\left(k+2\right)\psi(t)dt
\]
\[
 \leq C\|b\|_{BMO(\omega)}\sum_{k=0}^\infty\int_{a^{-(k+1)/\gamma}}^{a^{- k/\gamma}}\left(M_\omega f^r(s(t)x)\right)^{1/r}\left|\log |s(t)|\right|\psi(t)dt+2\int_0^1\left(M_\omega f^r(s(t)x)\right)^{1/r}\psi(t)dt
\]
\[
\leq C\|b\|_{BMO(\omega)}\int_0^1\left(M_\omega f^r(s(t)x)\right)^{1/r}\left(2+\left|\log\frac1{ |s(t)|}\right|\right)\psi(t)dt.
\]
\vskip12pt
Combining the estimates of $I_1,I_2$ and $I_3$, we obtain
\[
\frac1{\omega(B)}\int_B\left|U^b_{\psi,s}f(y)\right|\omega(y)dy\leq  C\|b\|_{BMO(\omega)}\int_0^1\left(M_\omega f^r(s(t)x)\right)^{1/r}\left(2+\left|\log\frac1{ |s(t)|}\right|\right)\psi(t)dt.
\]
\vskip12pt
Take the supremum over all $B$ such that $x\in B$, and take $L^p(\omega)$-norm of both side of the inequality above. We get
\begin{equation}\label{eq16}
\|M_\omega\left(U^b_{\psi,s}f\right)\|_{L^p(\omega)}\leq  C\|b\|_{BMO(\omega)}\left\|\int_0^1\left(M_\omega f^r(s(t)\cdot )\right)^{1/r}\left(2+\left|\log\frac1{ |s(t)|}\right|\right)\psi(t)dt\right\|_{L^p(\omega)}.
\end{equation}
The Minkowski inequality yields
\begin{align*}
&\qquad\|M_\omega\left(U^b_{\psi,s}\right)(\cdot)\|_{L^p(\omega)}\\
&\leq  C\|b\|_{BMO(\omega)}\left\|\int_0^1\left(M_\omega f^r(s(t)\cdot )\right)^{1/r}\left(2+\left|\log\frac1{ |s(t)|}\right|\right)\psi(t)dt\right\|_{L^p(\omega)}\\ 
&\leq  C\|b\|_{BMO(\omega)}\int_0^1\left(\int_{\mathbb R^n}\left(M_\omega f^r\left(s(t)x\right)\right)^{p/r}\omega(x)dx\right)^{1/p}\psi(t)\left(2+\left|\log\frac1{ |s(t)|}\right|\right)dt\\
&\leq  C\|b\|_{BMO(\omega)}\int_0^1\left(\int_{\mathbb R^n}\left(M_\omega f^r\left(x\right)\right)^{p/r}\omega(x)dx\right)^{1/p}|s(t)|^{-\frac{n+\alpha}p}\psi(t)\left(2+\left|\log\frac1{ |s(t)|}\right|\right)dt\\
&\leq  C\|b\|_{BMO(\omega)}\|f\|_{L^p(\omega)}\int_0^1|s(t)|^{-\frac{n+\alpha}p}\psi(t)\left(2+\left|\log\frac1{ |s(t)|}\right|\right)dt,\\
\end{align*}
where the last inequality is deduced by using lemma \ref{lem1c}. Since $U^b_{\psi,s}f\in L^1_{\text{loc}}(\omega)$, so by Lebesgue differential theorem $\left|U^b_{\psi,s}f(x)\right|\leq M_\omega\left(U^b_{\psi,s}f\right)(x)$ a.e. We arrive at
\[
\|U^b_{\psi,s}f\|_{L^p(\omega)} \leq  C\|b\|_{BMO(\omega)}\|f\|_{L^p(\omega)}\int_0^1|s(t)|^{-\frac{n+\alpha}p}\psi(t)\left(2+\left|\log\frac1{ |s(t)|}\right|\right)dt<\infty.
\]

This finishes the proof of part (b). $\hspace{6.4cm} \Box$

\begin{corollary}\label{theo4}
Let $\psi:[0,1]\to[0,\infty)$, $s:[0,1]\to\mathbb R$ be measurable functions, $\omega\in\mathcal W_\alpha$ has doubling property  and $1<p<\infty$. We also assume that there exists real constants $\beta,\gamma$, such that $\beta\gamma>0$ and $t^\beta\leq |s(t)|\leq t^\gamma$ for almost everywhere $t\in[0,1]$. 
\begin{itemize}
\item[{(a)}] If  $V^b_{\psi,s}$ is bounded on $L^p(\omega)$ for any $b\in BMO(\omega)$, then
\begin{equation}\label{eq17}
\begin{cases}
\int\limits_0^1|s(t)|^{n-\frac{n+\alpha}p}\cdot\left|\log\frac1{ |s(t)|}\right|\cdot\psi(t)dt<\infty,\vspace*{10pt}&\\
\int\limits_0^1|s(t)|^{n-\frac{n+\alpha}p}\cdot\psi(t)dt<\infty.&\\
\end{cases}
\end{equation}
\item[{(b)}] If $|s(t)|=t^\gamma$, and (\ref{eq17}) holds, then $V^b_{\psi,s}$ is bounded on $L^p(\omega)$ for any $b\in BMO(\omega)$ with operator norm
\[
\|V^b_{\psi,s}\|_{L^p(\omega)}\leq C_{\omega}\left(\int_0^1|s(t)|^{n-\frac{n+\alpha}p}\psi(t)\left(2+\left|\log\frac1{ |s(t)|}\right|\right)dt\right)\cdot\|b\|_{BMO(\omega)}.
\]
\end{itemize}
\end{corollary}

\bibliographystyle{amsplain}

\end{document}